\documentclass[psreqno,reqno,11pt]{amsart}
\usepackage{eurosym}
\usepackage{amssymb}
\usepackage{tikz}
\usepackage{amsmath}
\usepackage{tikz}
\usepackage{pgflibraryplotmarks}
\usepackage{wasysym}
\usepackage{stmaryrd}
\usepackage[english]{babel}

\usepackage{hyperref}

\theoremstyle{plain}
\newtheorem{theorem}{Theorem}[section]

\newtheorem{lemma}[theorem]{Lemma}

\theoremstyle{definition}

\newtheorem*{Acknowledgement}{Acknowledgment}

\begin{document}
\def\N{\mathbb{N}}
\def\Z{\mathbb{Z}}
\def\R{\mathbb{R}}
\def\C{\mathbb{C}}
\def\w{\omega}

\title[Characterizing Bounded Orthogonally Additive Polynomials]{
Characterizing Bounded Orthogonally Additive Polynomials on Vector Lattices}
\author{G. Buskes}
\address{Department of Mathematics, University of Mississippi, University,
MS 38677, USA}
\email{mmbuskes@olemiss.edu}
\author{C. Schwanke}
\address{Unit for BMI, North-West University, Private Bag X6001, Potchefstroom,
2520, South Africa}
\email{cmschwanke26@gmail.com}
\date{\today}

\begin{abstract}
We derive formulas for characterizing bounded orthogonally additive polynomials in two ways. Firstly, we prove that certain formulas for orthogonally additive
polynomials derived in \cite{Kusa} actually characterize them. Secondly, by employing
complexifications of the unique symmetric multilinear maps associated with orthogonally additive maps we derive new characterizing formulas.
\end{abstract}

\maketitle

\section{Introduction}\label{S:intro}

In \cite{Kusa}, Kusraeva uses the $n^{th}$ root mean power $\mathfrak{S}_{n}$ and the $n^{th}$ geometric mean $\mathfrak{G}_{n}$ in uniformly complete vector lattices (both defined by the Archimedean vector lattice functional calculus) to derive two interesting formulas in the theory of orthogonally additive homogeneous polynomials on Archimedean vector lattices. Indeed, Kusraeva proves that for a bounded orthogonally additive $n$-homogeneous polynomial $P$ on a uniformly complete Archimedean vector lattice $E$ with values in a convex bornological vector space $Y$ (with unique corresponding symmetric $n$-linear map $\check{P}$), the following equalities hold for all $r\in\mathbb{N}$ and all $f_{1},\dots,f_{\max\{n,r\}}\in E^+$:
\[
P(\mathfrak{S}_{n}(f_{1},\dots,f_{r}))=P(f_{1})+\dots+P(f_{r})
\]
and
\[
P(\mathfrak{G}_{n}(f_{1},\dots,f_{n}))=\check{P}(f_{1},\dots,f_{n}).
\]
In this note we show that these properties \textit{characterize} bounded orthogonally
additive polynomials. Additionally, by distinguishing between odd and even
polynomials, we provide additional characterizing formulae for bounded orthogonally
additive polynomials by using the complexification of the unique corresponding
symmetric multilinear maps. For the definition of $n^{th}$ root mean power and $n^{th}$
geometric mean via functional calculus in uniformly complete Archimedean vector lattices we refer to \cite{BusSch} and for the theory of functional calculus we refer to \cite{BusdPvR}. For further information on the theory of orthogonally additive polynomials on vector lattices we refer to \cite{BenLassLlav,BuBus,Loane2}.

We use this introduction to provide a couple of notational shortcuts that we
will use to facilitate our proofs. For the $n^{th}$ geometric mean $\mathfrak{G}_{n}$ on a uniformly complete Archimedean vector lattice $E$, a symmetric $n$-linear map $T$, and $f,g\in E$ we will write
\[
\mathfrak{G}_{n}(f^{n-k}g^{k})\text{,}%
\]
respectively%
\[
T(f^{n-k}g^{k})\text{,}%
\]
as a shorthand for the geometric mean of $n$ elements of $E$ of which $n-k$ are
equal to $f$ and $k$ are equal to $g$, respectively the image of $T$ when
$n-k$ entries are equal to $f$ and the other entries are equal to $g$. In particular, the expressions $\mathfrak{G}_n(f^n)$ and $T(f^n)$ are shorthands for the geometric mean, respectively the image of $T$, when all $n$ variables are equal to $f$.

Moreover, if $s,t,n\in\N$ satisfy $s+t=n-1$ and $T$ is a symmetric $n$-linear map on a uniformly complete vector lattice $E$ then for $f,g,z\in E$ we will write
\[
T(f^sg^tz)\text{,}%
\]
as a shorthand for the image of $T$ when $s$ entries are equal to $f$ and $t$ entries are equal to $g$ and one entry is equal to $z$. The shorthand for the geometric mean
\[
\mathfrak{G}_{n}(f^sg^tz)\text{,}%
\]
is used analogously.

Additionally, it will help the reader to keep in mind that, if $E$ equals the real numbers and $r\in\N$, we have for $x_1,\dots,x_{\max\{n,r\}}\in\R^+$ that
\[
\mathfrak{S}_{n}(x_{1},\dots,x_{r})=\left(
%TCIMACRO{\tsum \limits_{k=1}^{r}}%
%BeginExpansion
{\textstyle\sum\limits_{k=1}^{r}}
%EndExpansion
\left\vert x_{k}\right\vert ^{n}\right)  ^{\frac{1}{n}}
\]
and
\[
\mathfrak{G}_{n}(x_{1},\dots,x_{n})=\left(
%TCIMACRO{\tprod \limits_{k=1}^{n}}%
%BeginExpansion
{\textstyle\prod\limits_{k=1}^{n}}
%EndExpansion
\left\vert x_{k}\right\vert \right)^{\frac{1}{n}}.%
\]
Note that for a scalar $\alpha\in\R^+$ we have%
\[
\mathfrak{G}_{n}(\alpha
f^{n-k}g^{k})=\alpha^{1/n}\mathfrak{G}_{n}(f^{n-k}g^{k})%
\]
where $\mathfrak{G}_{n}(\alpha f^{k}g^{n-k})$ is a shorthand for $\mathfrak{G}_{n}%
(f^{k}g^{n-k})$ in which additionally the first entry is multiplied by
$\alpha$.

Given a uniformly complete Archimedean vector lattice $E$, we denote it's vector space complexification $E+iE$ by $E_\C$, for short. Since $E$ is uniformly complete, the expression
\[
|z|=\sup_{\theta\in[0,2\pi]}\{(\cos\theta)f+(\sin\theta)g\}
\]
is well defined for all $z=f+ig\in E_\C$ (with $f,g\in E$), and so $E_\C$ is an Archimedean complex vector lattice. As usual, for $f,g\in E$ and $z=f+ig\in E_\C$ we define the complex conjugate $\bar{z}=f-ig$.

In the proof of Theorem~\ref{T:even} we utilize the richer algebraic structure of complex vector lattices to obtain the converse of Kusraeva's theorem mentioned above. This requires us to extend symmetric multilinear maps defined on real vector lattices to corresponding symmetric multilinear maps defined on the vector space complexifications of these vector lattices. To this end, for a uniformly complete Archimedean real vector lattice $E$, a real vector space $V$, and a symmetric $n$-linear map $T\colon E\times\cdots\times E\to V$, we define the complex symmetric $n$-linear map $T_{\mathbb{C}}\colon E_\C\times\cdots\times E_\C\to V_\C$ by
\begin{equation*}
T_\C(f_{0}^{1}+if_{1}^{1},\dots,f_{0}^{n}+if_{1}^{n})=\sum\limits_{\epsilon_{k}\in\{0,1\}}i^{\sum\limits_{k=1}^{n}\epsilon_{k}}T(f_{\epsilon _{1}}^{1},\dots,f_{\epsilon_{n}}^{n})
\end{equation*}
for every $(f_{0}^{1}+if_{1}^{1},\dots,f_{0}^{n}+if_{1}^{n})\in E_\C\times\cdots\times E_\C$.

For more information on these complexifications of multilinear maps on vector lattices, we refer the reader to \cite{BusSch2}. For any unexplained terminology or basic results in vector lattice theory, we refer the reader to \cite{AB, LuxZan1, Zan2}.

\section{Main Results}\label{S:main}

The following two lemmas are needed for Theorem~\ref{T:even} and Theorem~\ref{T:odd}. We denote the algebraic dual of a vector space $V$ by $V^\#$.

\begin{lemma}\label{L:P=0}
Let $V$ be a real vector space. If $v_0,v_1,\dots,v_n\in V$ and $v_0+\lambda v_1+\dots+\lambda^nv_n=0$ for every $\lambda\in\R^+$ then $v_0=\dots=v_n=0$. 
\end{lemma}

\begin{proof}
The result for the case $V=\R$ follows from the fact that if the function $f\colon\R^+\to\R$ defined by $f(\lambda)=v_0+\lambda v_1+\dots+\lambda^nv_n$ is identically zero then for all of the higher order right derivatives $\partial^{(k)}_+f$, we also have $\partial_+^{(k)}f(\lambda)=0\ (\lambda\in\R^+)$. Next suppose $V$ is an arbitrary real vector space. Let $v_0,v_1,\dots,v_n\in V$, and suppose $v_0+\lambda v_1+\dots+\lambda^nv_n=0$ for every $\lambda\in\R^+$. Then for all $\phi\in V^\#$ we have $\phi(v_0)+\lambda\phi(v_1)+\dots+\lambda^n\phi(v_n)=0\quad (\lambda\in\R^+)$. But since the given result holds for $V=\R$ we have $\phi(v_0)=\dots=\phi(v_n)=0$ for all $\phi\in V^\#$. The result now follows from the fact that $V^\#$ separates the points of $V$.
\end{proof}

\begin{lemma}\label{L:|z|}
Let $E$ be a uniformly complete Archimedean vector lattice, let $f,g\in E$, and set $z=f+ig\in E_\C$.
\begin{itemize}
\item[(i)] If $n\in\mathbb{N}$ is even and $m=\frac{n}{2}$ then
\[
|z|=\mathfrak{S}_{n}\left\{\mathfrak{G}_n\left(f^{n}\right),\mathfrak{G}_n\left(mf^{n-2}g^2\right),\mathfrak{G}_n\left(\binom{m}{2}f^{n-4}g^4\right),\mathfrak{G}_n\left(\binom{m}{3}f^{n-6}g^6\right),\dots,\mathfrak{G}_n\left(g^n\right)\right\}.
\]
\item[(ii)] If $n\in\mathbb{N}$ is odd and $m=\frac{n-1}{2}$ then
\[
|z|=\mathfrak{S}_{n}\left\{\mathfrak{G}_n\left(f^{n-1}|z|\right),\mathfrak{G}_n\left(mf^{n-3}g^2|z|\right),\mathfrak{G}_n\left(\binom{m}{2}f^{n-5}g^4|z|\right),\dots,\mathfrak{G}_n\left(g^n|z|\right)\right\}.
\]
\end{itemize}
\end{lemma}

\begin{proof}
We only prove (1), leaving the similar proof of (2) to the reader. To this end, let $F$ be the vector sublattice of $E$ generated by
\[
f,g,\mathfrak{G}_n\left(f^{n}\right),\mathfrak{G}_n\left(f^{n-2}g^2\right),\mathfrak{G}_n\left(f^{n-4}g^4\right),\mathfrak{G}_n\left(f^{n-6}g^6\right),\dots,\mathfrak{G}_n\left(g^n\right),\ \text{and}
\]
\[
\mathfrak{S}_{n}\left\{\mathfrak{G}_n\left(f^{n}\right),\mathfrak{G}_n\left(mf^{n-2}g^2\right),\mathfrak{G}_n\left(\binom{m}{2}f^{n-4}g^4\right),\mathfrak{G}_n\left(\binom{m}{3}f^{n-6}g^6\right),\dots,\mathfrak{G}_n\left(g^n\right)\right\}.
\]
Let $\w:F\rightarrow\R$ be a vector lattice homomorphism. Using \cite[Theorem 3.7(1)]{BusSch} (second equality) and \cite[Theorem 3.11]{BusSch} (third equality), we have
\begin{align*}\label{eq:1}
\w(|z|)&=\w\left(\underset{\theta\in[0,2\pi]}{\sup}\{fcos\theta+g\sin\theta\}\right)=\w\left(\mathfrak{S}_2(f,g)\right)=\mathfrak{S}_2\bigl(\w(f),\w(g)\bigr)\\
&=\bigl(\w(f)^2+\w(g)^2\bigr)^\frac{1}{2}.
\end{align*}
Therefore,
\begin{equation}\label{eq:1}
\w(|z|)=\big((\w(f)^2+\w(g)^2)^{m}\big)^{\frac{1}{n}}.
\end{equation}
Note that
\begin{align*}
(\w(f)^2+\w(g)^2)^{m}&=\sum_{k=0}^{m}\binom{m}{k}\w(f)^{n-2k}\w(g)^{2k}=\sum_{k=0}^{m}\binom{m}{k}\big(\w(f)^{\frac{n-2k}{n}}\w(g)^{\frac{2k}{n}}\big)^n\\
&=\sum_{k=0}^{m}\binom{m}{k}\mathfrak{G}_n(\w(f)^{n-2k}\w(g)^{2k})^n=\sum_{k=0}^{m}\mathfrak{G}_n\left(\binom{m}{k}\w(f)^{n-2k}\w(g)^{2k}\right)^n.
\end{align*}
Using equation (\ref{eq:1}) we have
\begin{align*}
\w(|z|)&=\left(\sum_{k=0}^{m}\mathfrak{G}_n\left(\binom{m}{k}\w(f)^{n-2k}\w(g)^{2k}\right)^n\right)^{\frac{1}{n}}\\
&=\mathfrak{S}_{n}\left\{\mathfrak{G}_n\bigl(\w(f)^{n}\bigr),\mathfrak{G}_n\bigl(m\w(f)^{n-2}\w(g)^2\bigr),\mathfrak{G}_n\left(\binom{m}{2}\w(f)^{n-4}\w(g)^4\right),\dots,\mathfrak{G}_n\bigl(\w(g)^n\bigr)\right\}\\
&=\mathfrak{S}_{n}\left\{\w\bigl(\mathfrak{G}_n(f^{n})\bigr),\w\bigl(\mathfrak{G}_n(mf^{n-2}g^2)\bigr),\omega\left(\mathfrak{G}_n\left(\binom{m}{2}\w(f)^{n-4}g^4\right)\right),\dots,\w\bigl(\mathfrak{G}_n(g)^n)\bigr)\right\}\\
&=\w\left(\mathfrak{S}_{n}\left\{\mathfrak{G}_n\left(f^{n}\right),\mathfrak{G}_n\left(mf^{n-2}g^2\right),\mathfrak{G}_n\left(\binom{m}{2}f^{n-4}g^4\right),\dots,\mathfrak{G}_n\left(g^n\right)\right\}\right).
\end{align*}
Since the set of all nonzero vector lattice homomorphisms from $F$ into $\mathbb{R}$ separates the points of $F$ (see \cite[(ii) on page 526 and Theorem 2.2]{BusvR3}), we have
\[
|z|=\mathfrak{S}_{n}\left\{\mathfrak{G}_n\left(f^{n}\right),\mathfrak{G}_n\left(mf^{n-2}g^2\right),\mathfrak{G}_n\left(\binom{m}{2}f^{n-4}g^4\right),\mathfrak{G}_n\left(\binom{m}{3}f^{n-6}g^6\right),\dots,\mathfrak{G}_n\left(g^n\right)\right\}.
\]
\end{proof}

We proceed to our first of two main results. The implication (iv)$\implies$(i) of the following theorem is contained in \cite[Proposition 4.38]{Loane2}. Since \cite{Loane2} is not readily available, we record the details of the proof.

\begin{theorem}\label{T:even}
Let $E$ be a uniformly complete vector lattice, let $Y$ be a convex bornological space, let $n,r\in\mathbb{N}$ with $n$ even, let $P$ be a bounded $n$-homogeneous polynomial, and let $\check{P}$ be the unique symmetric $n$-linear map associated with $P$. The following are equivalent.
\begin{itemize}
\item[(i)] $P$ is orthogonally additive.
\item[(ii)] For $f_{1},\dots,f_{\max\{n,r\}}\in E^{+}$, the following equalities hold:
\[
P\bigl(\mathfrak{S}_{n}(f_1,\dots,f_r)\bigr)=P(f_1)+\cdots+P(f_r);
\]
\[
P\bigl(\mathfrak{G}_n(f_1,\dots,f_n)\bigr)=\check{P}(f_1,\dots,f_n).
\]
\item[(iii)] $P(|z|)=\check{P}_{\C}(z^{\frac{n}{2}}(\bar{z})^{\frac{n}{2}})$ for every $z\in E_{\C}$.
\item[(iv)] $\check{P}(f^{n-k},g^{k})=0$ for every $k\in\lbrace 1,\dots,n-1\rbrace$ whenever $f\perp g$.
\end{itemize}
\end{theorem}

\begin{proof}
(i)$\implies$(ii) This is the content of the main theorem of \cite{Kusa}.

\medskip
\noindent
(ii)$\implies$(iii) Suppose that for all $f_{1},\dots,f_{\max\{n,r\}}\in E^{+}$, we have
\[
P\bigl(\mathfrak{S}_{n}(f_1,\dots,f_r)\bigr)=P(f_1)+\cdots+P(f_r),\ \text{and}
\]
\[
P\bigl(\mathfrak{G}_n(f_1,\dots,f_n)\bigr)=\check{P}(f_1,\dots,f_n).
\]
Let $z=f+ig\in E_{\C}$. For convenience, we write $m=\frac{n}{2}$. Using Lemma~\ref{L:|z|}(i), we obtain
\begin{align*}
P(|z|)&=P\left(\mathfrak{S}_{n}\left\{\mathfrak{G}_n\left(f^{n}\right),\mathfrak{G}_n\left(mf^{n-2}g^2\right),\mathfrak{G}_n\left(\binom{m}{2}f^{n-4}g^4\right),\dots,\mathfrak{G}_n\left(g^n\right)\right\}\right)\\
&=\sum_{k=0}^{m}P\left(\mathfrak{G}_n\left(\binom{m}{k}f^{n-2k}g^{2k}\right)\right)=\sum_{k=0}^{m}P\left(\binom{m}{k}^{\frac{1}{n}}\mathfrak{G}_n\left(f^{n-2k}g^{2k}\right)\right)\\
&=\sum_{k=0}^{m}\binom{m}{k}P\left(\mathfrak{G}_n\left(f^{n-2k}g^{2k}\right)\right)=\sum_{k=0}^{m}\binom{m}{k}\check{P}\left(f^{n-2k}g^{2k}\right)\\
&=\check{P}_{\C}(z^{m}(\bar{z})^m)=\check{P}_{\C}(z^{\frac{n}{2}}(\bar{z})^{\frac{n}{2}}).
\end{align*}

\medskip
\noindent
(iii)$\implies$(iv) Suppose that $P(|z|)=\check{P}_{\C}(z^{\frac{n}{2}}\bar{z}^{\frac{n}{2}})$ for every $z\in E_{\C}$. Let $f,g\in E^{+}$, and assume $f\perp g$. By \cite[Theorem 14.1(i)]{LuxZan1}, we have $f+g=|f+g|=|f-g|$. Moreover, the string
\[
f+g=||f|-|g||\leq|f+ig|\leq f+g
\]
implies that $|f+ig|=f+g$. We have
\begin{align*}
P(|f+g|)&=\check{P}\bigl((f+g)^n\bigr)\\
&=\sum\limits_{k=0}^{n}\binom{n}{k}\check{P}(f^{n-k}g^{k}),
\end{align*}
while by our assumption,
\begin{align*}
P(|f-g|)&=\check{P}\bigl((f-g)^n\bigr)\\
&=\sum\limits_{k=0}^{n}\binom{n}{k}(-1)^{k}\check{P}(f^{n-k}g^{k}).
\end{align*}
Therefore,
\begin{align*}
\sum\limits_{k=0}^{n}\binom{n}{k}\check{P}(f^{n-k}g^{k})=\sum\limits_{k=0}^{n}\binom{n}{k}(-1)^{k}\check{P}(f^{n-k}g^{k}),
\end{align*}
and thus
\begin{align*}
\sum\limits_{k=0}^{n}(1-(-1)^{k})\binom{n}{k}\check{P}(f^{n-k}g^{k})=2\sum\limits_{k=0}^{\frac{n-2}{2}}\binom{n}{2k+1}\check{P}(f^{n-(2k+1)}g^{2k+1})=0.
\end{align*}
But then for every $\lambda\in\R^{+}$ we have
\begin{align*}
\sum\limits_{k=0}^{\frac{n-2}{2}}\binom{n}{2k+1}\lambda^{2k+1}\check{P}(f^{n-(2k+1)}g^{2k+1})=0.
\end{align*}
It follows from Lemma~\ref{L:P=0} that $\check{P}(f^{n-(2k+1)}g^{2k+1})=0$ for every $k\in\lbrace 0,\dots,\frac{n-2}{2}\rbrace$. Hence,
\begin{align*}
\check{P}\bigl((f+g)^n\bigr)&=\sum\limits_{k=0}^{n}\binom{n}{k}\check{P}(f^{n-k},g^{k})=\sum\limits_{k=0}^{\frac{n}{2}}\binom{n}{2k}\check{P}(f^{n-2k}g^{2k}).
\end{align*}

On the other hand, $|f+ig|=f+g$ implies that
\begin{align*}
\check{P}(|f+ig|^n)=\sum\limits_{k=0}^{\frac{n}{2}}\binom{n}{2k}\check{P}(f^{n-2k}g^{2k}).
\end{align*}
Moreover, by assumption, we have
\begin{align*}
\check{P}(|f+ig|^n)=\check{P}_{\C}\bigl((f+ig)^{\frac{n}{2}}(f-ig)^{\frac{n}{2}}\bigr),
\end{align*}
and the multi-binomial theorem yields
\begin{align*}
\check{P}_{\C}\bigl((f+ig)^{\frac{n}{2}}(f-ig)^{\frac{n}{2}}\bigr)=\sum\limits_{k=0}^{\frac{n}{2}}\binom{\frac{n}{2}}{k}\check{P}(f^{n-2k}g^{2k}).
\end{align*}
Therefore, we obtain
\begin{align*}
\sum\limits_{k=0}^{\frac{n}{2}}\binom{n}{2k}\check{P}(f^{n-2k}g^{2k})=\sum\limits_{k=0}^{\frac{n}{2}}\binom{\frac{n}{2}}{k}\check{P}(f^{n-2k}g^{2k}).
\end{align*}
In fact, for every $\lambda\in\R^{+}$ we have
\begin{align*}
\sum\limits_{k=0}^{\frac{n}{2}}\left(\binom{n}{2k}-\binom{\frac{n}{2}}{k}\right)\lambda^{2k}\check{P}(f^{n-2k}g^{2k})=0.
\end{align*}
It follows from Lemma~\ref{L:P=0} that $\check{P}(f^{n-2k}g^{2k})=0$ for every $k\in\lbrace 1,\dots,\frac{n}{2}\rbrace$. We conclude that $\check{P}(f^{n-k}g^{k})=0$ for each $k\in\lbrace 1,\dots,n-1\rbrace$. The result now follows from the $n$-linearity of $\check{P}$ and the decomposition $f=f^+-f^-$ for any $f\in E$.

\medskip
\noindent
(iv)$\implies$(i) Suppose that $\check{P}(f^{n-k}g^{k})=0$ for every $k\in\lbrace 1,\dots,n-1\rbrace$ whenever $f$ and $g$ are disjoint. Let $f,g\in E$, and assume $f\perp g$. Then
\begin{align*}
P(f+g)&=P(f)+P(g)+\sum\limits_{k=1}^{n-1}\binom{n}{k}\check{P}(f^{n-k}g^{k})\\
&=P(f)+P(g).
\end{align*}
Hence $P$ is orthogonally additive.
\end{proof}

The following theorem is an analogue of Theorem~\ref{T:even} for odd $n\in\N$.

\begin{theorem}\label{T:odd}
Let $E$ be a uniformly complete vector lattice, let $Y$ be a convex bornological space, let $n,r\in\mathbb{N}$ with $n$ odd, let $P$ be a bounded $n$-homogeneous polynomial, and let $\check{P}$ be the unique symmetric $n$-linear map associated with $P$. The following are equivalent.
\begin{itemize}
\item[(i)] $P$ is orthogonally additive.
\item[(ii)] For $f_{1},\dots,f_{\max\{n,r\}}\in E^{+}$, the following equalities hold:
\[
P\bigl(\mathfrak{S}_{n}(f_1,\dots,f_r)\bigr)=P(f_1)+\cdots+P(f_r);
\]
\[
P\bigl(\mathfrak{G}_n(f_1,\dots,f_n)\bigr)=\check{P}(f_1,\dots,f_n).
\]
\item[(iii)] $P(|z|)=\check{P}_{\C}(z^{\frac{n-1}{2}}(\bar{z})^{\frac{n-1}{2}}|z|)$ for every $z\in E_{\C}$.
\item[(iv)] $\check{P}(f^{n-k}g^{k})=0$ for every $k\in\lbrace 1,\dots,n-1\rbrace$ whenever $f\perp g$.
\end{itemize}
\end{theorem}

\begin{proof}
(i)$\implies$(ii) See the main theorem of \cite{Kusa}.

\medskip
\noindent
(ii)$\implies$(iii) Let $z=f+ig\in E_{\C}$. For convenience, we write $m=\frac{n-1}{2}$. Using Lemma~\ref{L:|z|}(ii), we have
\begin{align*}
P(|z|)&=P\left(\mathfrak{S}_{n}\left\{\mathfrak{G}_n\left(f^{n-1}|z|\right),\mathfrak{G}_n\left(mf^{n-3}g^2|z|\right),\mathfrak{G}_n\left(\binom{m}{2}f^{n-5}g^4|z|\right),\dots,\mathfrak{G}_n\left(g^n|z|\right)\right\}\right)\\
&=\sum_{k=0}^{m}P\left(\mathfrak{G}_n\left(\binom{m}{k}f^{n-1-2k}g^{2k}|z|\right)\right)=\sum_{k=0}^{m}P\left(\binom{m}{k}^{\frac{1}{n}}\mathfrak{G}_n\left(f^{n-1-2k}g^{2k}|z|\right)\right)\\
&=\sum_{k=0}^{m}\binom{m}{k}P\left(\mathfrak{G}_n\left(f^{n-1-2k}g^{2k}|z|\right)\right)=\sum_{k=0}^{m}\binom{m}{k}\check{P}\left(f^{n-1-2k}g^{2k}|z|\right)\\
&=\check{P}_{\C}(z^{m}(\bar{z})^m|z|)=\check{P}_{\C}(z^{\frac{n-1}{2}}(\bar{z})^{\frac{n-1}{2}}|z|).
\end{align*}

\medskip
\noindent
(iii)$\implies$(iv) Suppose that $P(|z|)=\check{P}_{\C}(z^{\frac{n-1}{2}}(\bar{z})^{\frac{n-1}{2}}|z|)$ for every $z\in E_\C$. Let $f,g\in E^{+}$ be such that $f\perp g$. Then
\begin{align*}
P(|f+g|)&=\check{P}\bigl((f+g)^{n}\bigr)\\
&=\sum\limits_{k=0}^{n}\binom{n}{k}\check{P}(f^{n-k}g^{k}).
\end{align*}
Moreover, by assumption
\begin{align*}
P(|f-g|)&=\check{P}\bigl((f-g)^{n-1}|f-g|\bigr)\\
&=\check{P}\bigl((f-g)^{n-1}(f+g)\bigr)\\
&=\sum\limits_{k=0}^{n-1}\binom{n-1}{k}(-1)^{k}\check{P}(f^{n-1-k}g^{k}(f+g))\\
&=\sum\limits_{k=0}^{n-1}\binom{n-1}{k}(-1)^{k}\check{P}(f^{n-k}g^{k})+\sum\limits_{k=0}^{n-1}\binom{n-1}{k}(-1)^{k}\check{P}(f^{n-(k+1)}g^{k+1})\\
&=\sum\limits_{k=0}^{n-1}\binom{n-1}{k}(-1)^{k}\check{P}(f^{n-k}g^{k})-\sum\limits_{k=1}^{n}\binom{n-1}{k-1}(-1)^{k}\check{P}(f^{n-k}g^{k})\\
&=\check{P}(f^n)+\check{P}(g^n)+\sum\limits_{k=1}^{n-1}\left(\binom{n-1}{k}-\binom{n-1}{k-1}\right)(-1)^{k}\check{P}(f^{n-k}g^{k}).
\end{align*}
Thus
\begin{align*}
\sum\limits_{k=0}^{n}\binom{n}{k}\check{P}(f^{n-k}g^{k})&=\check{P}(f^n)+\check{P}(g^n)+\sum\limits_{k=1}^{n-1}\left(\binom{n-1}{k}-\binom{n-1}{k-1}\right)(-1)^{k}\check{P}(f^{n-k}g^{k}).
\end{align*}
Hence
\begin{align*}
0&=\sum\limits_{k=0}^{n}\binom{n}{k}\check{P}(f^{n-k}g^{k})-\check{P}(f^n)-\check{P}(g^n)-\sum\limits_{k=1}^{n-1}\left(\binom{n-1}{k}-\binom{n-1}{k-1}\right)(-1)^{k}\check{P}(f^{n-k}g^{k})\\
&=\sum\limits_{k=1}^{n-1}\Biggl(\binom{n}{k}-\left(\binom{n-1}{k}-\binom{n-1}{k-1}\right)(-1)^{k}\Biggr)\check{P}(f^{n-k}g^{k})\\
&=2\sum\limits_{k=1}^{n-1}\binom{n-1}{\psi(k)}\check{P}(f^{n-k}g^{k}),
\end{align*}
where
\[
\begin{array}{ccc}
\psi(k)=\begin{cases} k &\mbox{if } k\ \text{is odd}\\
k-1 & \mbox{if } k\ \text{is even} \end{cases}\ (k\in\{1,\dots,n-1\}).
\end{array}
\]
Therefore,
\[
\sum\limits_{k=1}^{n-1}\binom{n-1}{\psi(k)}\check{P}(f^{n-k}g^{k})=0.
\]
But then for every $\lambda\in\R^{+}$ we have
\[
\sum\limits_{k=1}^{n-1}\lambda^{k}\binom{n-1}{\psi(k)}\check{P}(f^{n-k}g^{k})=0.
\]
It follows from Lemma~\ref{L:P=0} that $\check{P}(f^{n-k}g^{k})=0$ for every $k\in\lbrace 1,\dots,n-1\rbrace$. Using the $n$-linearity of $\check{P}$ and the decomposition $f=f^+-f^-$ for any $f\in E$, we obtain the desired result.

\medskip
\noindent
(iv)$\implies$(i) This proof is the same as in Theorem~\ref{T:even}.
\end{proof}

We conclude by pointing out that the implications (ii)$\implies$(iii), (iii)$\implies$(iv), and (iv)$\implies$(i) in Theorems \ref{T:even} and \ref{T:odd} still hold when $Y$ is a general real vector space and when $P$ is not necessarily bounded.

\begin{Acknowledgement}
	This research was partially supported by the Claude Leon Foundation and by the DST-NRF Centre of Excellence in Mathematical and Statistical Sciences (CoE-MaSS) (second author). Opinions expressed and conclusions arrived at are those of the authors and are not necessarily to be attributed to the CoE-MaSS.
\end{Acknowledgement}

\bibliography{cboapvl}

\providecommand{\bysame}{\leavevmode\hbox to3em{\hrulefill}\thinspace}
\providecommand{\MR}{\relax\ifhmode\unskip\space\fi MR }
% \MRhref is called by the amsart/book/proc definition of \MR.
\providecommand{\MRhref}[2]{%
  \href{http://www.ams.org/mathscinet-getitem?mr=#1}{#2}
}
\providecommand{\href}[2]{#2}
\begin{thebibliography}{10}

\bibitem{AB}
C.D. Aliprantis and O.~Burkinshaw, \emph{Positive {O}perators}, Academic Press,
  Orlando, 1985.

\bibitem{BenLassLlav}
Y.~Benyamini, S.~Lassalle, and J.G. Llavona, \emph{Homogeneous orthogonally
  additive polynomials on {B}anach lattices}, Bull. London Math. Soc.
  \textbf{38} (2006), no.~3, 459--469.

\bibitem{BuBus}
Q.~Bu and G.~Buskes, \emph{Polynomials on {B}anach lattices and positive tensor
  products}, J. Math. Anal. Appl. \textbf{388} (2012), no.~2, 845--862.

\bibitem{BusdPvR}
G.~Buskes, B.~de~Pagter, and A.~van Rooij, \emph{Functional calculus on {R}iesz
  spaces}, Indag. Math. (N.S.) \textbf{2} (1991), no.~4, 423--436.

\bibitem{BusSch2}
G.~Buskes and C.~Schwanke, \emph{Complex vector lattices via functional
  completions}, J. Math. Anal. Appl. \textbf{434} (2016), no.~2, 1762--1778.

\bibitem{BusSch}
\bysame, \emph{Functional completions of {A}rchimedean vector lattices},
  Algebra Universalis \textbf{76} (2016), no.~1, 53--69.

\bibitem{BusvR3}
G.~Buskes and A.~van Rooij, \emph{Small {R}iesz spaces}, Math. Proc. Cambridge
  Philos. Soc. \textbf{105} (1989), no.~3, 523--536.

\bibitem{Kusa}
Z.~A. Kusraeva, \emph{Homogeneous polynomials, power means and geometric means
  in vector lattices}, Vladikavkaz. Mat. Zh. \textbf{16} (2014), no.~4, 49--53.

\bibitem{Loane2}
J.~Loane, \emph{Polynomials on {R}iesz {S}paces}, Ph.D. thesis, Galway, 2007.

\bibitem{LuxZan1}
W.~A.~J. Luxemburg and A.~C. Zaanen, \emph{{R}iesz {S}paces {V}ol. {I}},
  North-Holland Publishing Co., Amsterdam-London; American Elsevier Publishing
  Co., New York, 1971.

\bibitem{Zan2}
A.C. Zaanen, \emph{Riesz {S}paces {II}}, North-Holland Mathematical Library,
  vol.~30, North-Holland Publishing Co., Amsterdam, 1983.

\end{thebibliography}
\bibliographystyle{amsplain}

\end{document}